\documentclass[11pt]{article}

\usepackage{amssymb, amsmath, amsthm, graphicx}
\usepackage[left=1in,top=1in,right=1in]{geometry}

\date{}

\theoremstyle{plain}
      \newtheorem{theorem}{Theorem}[section]
      \newtheorem{lemma}[theorem]{Lemma}

      \newtheorem{corollary}[theorem]{Corollary}

\theoremstyle{definition}

\theoremstyle{remark}

	\newcommand{\RR}{{\mathbb R}}

\def\twr{\mbox{\rm twr}}

\title{Semi-algebraic Ramsey numbers}
\author{ Andrew Suk\thanks{University of Illinois at Chicago, Chicago, IL, USA.  Supported by NSF grant DMS-1500153. Email: {\tt suk@math.uic.edu}.}}

\begin{document}

\maketitle

\begin{abstract}

Given a finite point set $P \subset \RR^d$, a $k$-ary semi-algebraic relation $E$ on $P$ is a set of $k$-tuples of points in $P$ determined by a finite number of polynomial equations and inequalities in $kd$ real variables. The description complexity of such a relation is at most
$t$ if the number of polynomials and their degrees are all bounded by $t$.  The Ramsey number $R^{d,t}_k(s,n)$ is the minimum $N$ such that any $N$-element point set $P$ in $\RR^d$ equipped with a $k$-ary semi-algebraic relation $E$ of complexity at most $t$ contains $s$ members such that every $k$-tuple induced by them is in $E$ or $n$ members such that every $k$-tuple induced by them is not in $E$.

We give a new upper bound for $R^{d,t}_k(s,n)$ for $k\geq 3$ and $s$ fixed.  In particular, we show that for fixed integers $d,t,s$

 $$R^{d,t}_3(s,n)  \leq 2^{n^{o(1)}},$$

\noindent establishing a subexponential upper bound on $R^{d,t}_3(s,n)$.  This improves the previous bound of $2^{n^{C_1}}$ due to Conlon, Fox, Pach, Sudakov, and Suk where $C_1$ depends on $d$ and $t$, and improves upon the trivial bound of $2^{n^{C_2}}$ which can be obtained by applying classical Ramsey numbers where $C_2$ depends on $s$.   As an application, we give new estimates for a recently studied Ramsey-type problem on hyperplane arrangements in $\RR^d$.  We also study multi-color Ramsey numbers for triangles in our semi-algebraic setting, achieving some partial results.

\end{abstract}

\section{Introduction}

\textbf{Classical Ramsey numbers.}  A $k$-uniform hypergraph $H = (P,E)$ consists of a vertex set $P$ and an edge set $E\subset {P\choose k}$, which is a collection of subsets
of $P$ of order $k$.  The Ramsey number $R_k(s,n)$ is the minimum integer $N$ such that every $k$-uniform hypergraph on $N$ vertices contains either $s$ vertices such that every $k$-tuple induced by them is an edge, or contains $n$ vertices such that every $k$-tuple induced by them is not an edge.

Due to its wide range of applications in logic, number theory, analysis, and geometry, estimating Ramsey numbers has become one of the most central problems in combinatorics.  For \emph{diagonal} Ramsey numbers, i.e.~when $s = n$, the best known lower and upper bounds for $R_k(n,n)$ are of the form\footnote{We write $f(n) = O(g(n))$ if $|f(n)| \leq c|g(n)|$ for some fixed constant $c$ and for all $n \geq 1$; $f(n) = \Omega(g(n))$ if $g(n) = O(f(n))$; and $f(n) = \Theta(g(n))$ if both $f(n) = O(g(n))$ and $f(n) = \Omega(g(n))$ hold.  We write $f(n) = o(g(n))$ if for every positive $\epsilon > 0$ there exists a constant $n_0$ such that $|f(n)| \leq \epsilon |g(n)|$ for all $n \geq n_0$.} $R_2(n,n) = 2^{\Theta(n)}$, and for $k \geq 3$,

$$\twr_{k-1}(\Omega(n^2)) \leq R_k(n,n) \leq \twr_k(O(n)),$$

\noindent where the tower function $\twr_k(x)$ is defined by $\twr_1(x) = x$ and $\twr_{i + 1} = 2^{\twr_i(x)}$ (see \cite{es,erdos2,erdos3,rado}).   Erd\H os, Hajnal, and Rado \cite{erdos3} conjectured that $R_k(n,n)=\twr_k(\Theta(n))$, and Erd\H os offered a \$500 reward for a proof.  Despite much attention over the last 50 years, the exponential gap between the lower and upper bounds for $R_k(n,n)$, when $k\geq 3$, remains unchanged.

The \emph{off-diagonal} Ramsey numbers, i.e.~$R_k(s,n)$ with $s$ fixed and $n$ tending to infinity, have also been extensively studied.  Unlike $R_k(n,n)$, the lower and upper bounds for $R_k(s,n)$ are much more comparable.  It is known \cite{AKS,Kim,B,BK} that $R_2(3,n) =\Theta(n^2/\log n)$ and, for fixed $s > 3$
\begin{equation}\label{offgraph}
\Omega\left( n^{\frac{s+1}{2} - \epsilon} \right) \leq R_2(s,n) \leq O\left(n^{s-1}\right),\end{equation}

\noindent  where $\epsilon > 0$ is an arbitrarily small constant.  Combining the upper bound in (\ref{offgraph}) with the results of Erd\H os, Hajnal, and Rado \cite{rado,erdos3} demonstrates that

\begin{equation}\label{rado2} \twr_{k-1}(\Omega(n))\leq  R_k(s,n) \leq \twr_{k-1}(O(n^{2s-4})),
\end{equation}

\noindent for $k \geq 3$ and $s \geq 2^k$.  See Conlon, Fox, and Sudakov \cite{conlon} for a recent improvement.
\medskip

\noindent \textbf{Semi-algebraic setting.}  In this paper, we continue a sequence of recent works on Ramsey numbers for $k$-ary semi-algebraic relations $E$ on $\RR^d$ (see \cite{bukh,matousek,suk,sukotd}).  Before we give its precise definition, let us recall two classic Ramsey-type theorems of Erd\H os and Szekeres.

\begin{theorem}[\cite{es}]\label{mono}
For $N = (s-1)(n-1) + 1$, let $P = (p_1,\ldots,p_N) \subset \mathbb{R}$ be a sequence of $N$ distinct real numbers.  Then $P$ contains either an increasing subsequence of length $s$, or a decreasing subsequence of length $n$.

\end{theorem}

\noindent In fact, there are now at least 6 different proofs of Theorem \ref{mono} (see \cite{steele}). The other well-known result from \cite{es} is the following theorem, which is often referred to as the Erd\H os-Szekeres cups-caps theorem.   Let $X$ be a finite point set in the plane in general position.\footnote{No two members share the same $x$-coordinate, and no three members are collinear.}  We say that $X = (p_{i_1},\ldots,p_{i_s})$ forms an \emph{$s$-cup} (\emph{$s$-cap}) if $X$ is in convex position\footnote{Forms the vertex set of a convex $s$-gon.} and its convex hull is bounded above (below) by a single edge.  See Figure~\ref{cupcap}.

\begin{figure}
\begin{center}
\includegraphics[width=280pt]{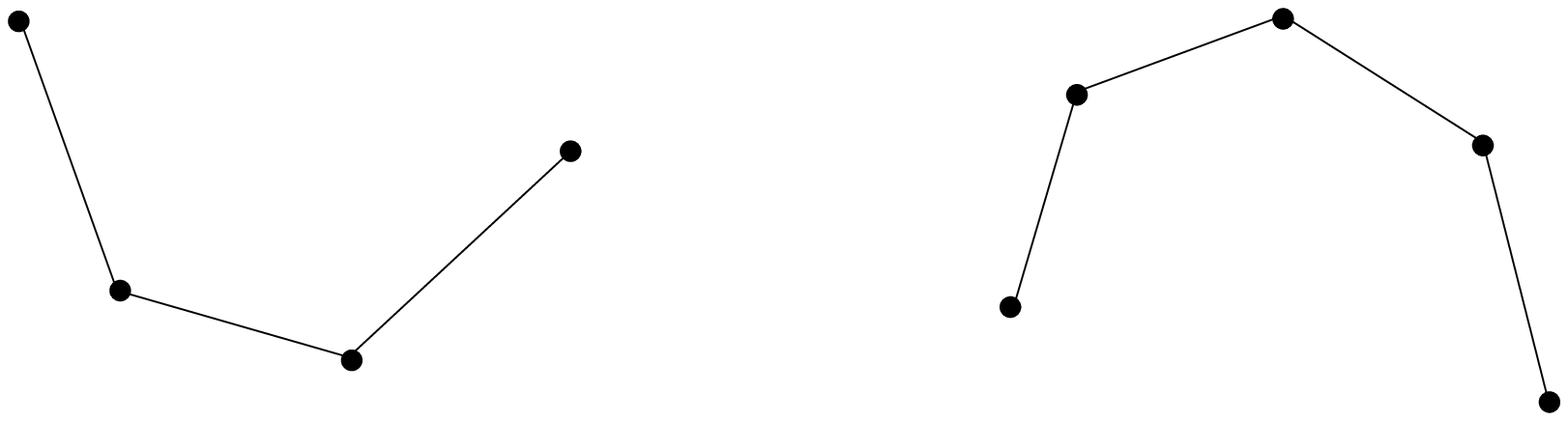}
  \caption{A 4-cup and a 5-cap.}\label{cupcap}
 \end{center}
\end{figure}

\begin{theorem}[\cite{es}]\label{mono2}
For $N =  {n  + s- 4\choose s-2} + 1$, let $P = (p_1,\ldots,p_N)$ be a sequence of $N$ points in the plane in general position.  Then $P$ contains either an $s$-cup or an $n$-cap.
\end{theorem}

Theorems \ref{mono} and \ref{mono2} can be generalized using the following semi-algebraic framework.  Let $P = \{p_1,\ldots,p_N\}$ be a sequence of $N$ points in $\RR^d$.  Then we say that $E\subset {P\choose k}$ is a \emph{semi-algebraic relation} on $P$ with \emph{complexity} at most $t$ if there are $t$ polynomials $f_1,\ldots,f_t \in \RR[x_1,\ldots,x_{kd}]$ of degree at most $t$, and a Boolean function $\Phi$ such that, for $1 \leq i_1 < \cdots < i_k \leq N$,

$$(p_{i_1},\ldots,p_{i_k}) \in E \hspace{.5cm}\Leftrightarrow \hspace{.5cm} \Phi(f_1(p_{i_1},\ldots,p_{i_k}) \geq 0,\ldots, f_t(p_{i_1},\ldots,p_{i_k}) \geq 0) = 1.$$

\noindent   We say that the relation $E\subset {P\choose k}$ is \emph{symmetric} if $(p_{i_1},\ldots,p_{i_k}) \in E$ iff for all permutation $\pi$,

$$ \Phi(f_1(p_{\pi(i_1)},\ldots,p_{\pi(i_k)}) \geq 0,\ldots, f_t(p_{\pi(i_1)},\ldots,p_{\pi(i_k)}) \geq 0) = 1.$$

Point sets $P\subset \RR^d$ equipped with a $k$-ary semi-algebraic relation $E\subset{P\choose k}$ are often used to model problems in discrete geometry, where the dimension $d$, uniformity $k$, and complexity $t$ are considered fixed but arbitrarily large constants.  Since we can always make any relation $E$ symmetric by increasing its complexity to $t' = t'(k,d,t)$, we can therefore simplify our presentation by only considering symmetric relations.

Let $R_{k}^{d,t}(s,n)$ be the minimum integer $N$ such that every $N$-element point set $P$ in $\RR^d$ equipped with a $k$-ary (symmetric) semi-algebraic relation $E\subset{P\choose k}$, which has complexity at most $t$, contains $s$ points such that every $k$-tuple induced by them is in $E$, or contains $n$ points such that no $k$-tuple induced by them is in $E$.  Alon, Pach, Pinchasi, Radoi\v{c}i\'c, and Sharir \cite{noga} showed that for $k = 2$, we have

\begin{equation}\label{semi2}R_2^{d,t}(n,n) \leq n^C,\end{equation}

\noindent where $C = C(d,t)$.  Roughly speaking, $C \approx t{d + t\choose t}$.   Conlon, Fox, Pach, Sudakov, and Suk showed that one can adapt the Erd\H os-Rado argument in \cite{rado} and establish the following recursive formula for $R_k^{d,t}(s,n).$

\begin{theorem}[\cite{suk}]\label{semirec}

Set $M = R^{d,t}_{k-1}(s-1,n-1)$.  Then for every $k \geq 3$,

$$R^{d,t}_{k}(s,n)  \leq 2^{C_1M\log M},$$

\noindent where $C_1=C_1(k,d,t)$.
\end{theorem}

\noindent Together with (\ref{semi2}) we have $R_k^{d,t}(n,n) \leq \twr_{k-1}(n^C)$, giving an exponential improvement over the Ramsey numbers for general $k$-uniform hypergraphs.  Conlon et al.~\cite{suk} also gave a construction of a geometric example that provides a $\twr_{k-1}(\Omega(n))$ lower bound, demonstrating that $R^{d,t}_k(n,n)$ does indeed grow as a $(k-1)$-fold exponential tower in $n$.

However, off-diagonal Ramsey numbers for semi-algebraic relations are much less well understood.  The best known upper bound for $R^{d,t}_k(s,n)$ is essentially the trivial bound $$R_k^{d,t}(s,n) \leq \min\left\{R^{d,t}_k(n,n), R_k(s,n)\right\}.$$  The crucial case is when $k = 3$, since any significant improvement on estimating $R_3^{d,t}(s,n)$ could be used with Theorem~\ref{semirec} to obtain a better bound for $R^{d,t}_k(s,n)$, for $k \geq 4$.  The trivial bound above implies that

\begin{equation}\label{triv}R^{d,t}_3(s,n) \leq   2^{n^C},\end{equation}

\noindent where $C = \min\{ C_1(d,t), C_2(s)\}$.

The main difficulty in improving (\ref{triv}) is that the Erd\H os-Rado upper bound argument \cite{rado} will not be effective.  Roughly speaking, the Erd\H os-Rado argument reduces the problem from 3-uniform hypergraphs to graphs, producing a recursive formula similar to Theorem \ref{semirec}.  This approach has been used repeatedly by many researchers to give upper bounds on Ramsey-type problems arising in triple systems \cite{conlon,suk,sukotd,dhruv2}.  However, it is very unlikely that any variant of the Erd\H os-Rados upper bound argument will establish a subexponential upper bound for $R^{d,t}_3(s,n)$.

With a more novel approach, our main result establishes the following improved upper bound for $R^{d,t}_3(s,n)$, showing that the function $R^{d,t}_3(s,n)$ is indeed subexponential in $n$.

\begin{theorem}\label{main}
For fixed integers $d,t\geq 1$ and $s\geq 4$, we have $R^{d,t}_3(s,n) \leq  2^{n^{o(1)}}.$  More precisely

$$R^{d,t}_3(s,n) \leq  2^{2^{c\sqrt{\log n \log\log n}}},$$

\noindent where $c = c(d,t,s)$.
\end{theorem}

\noindent Let us remark that in dimension 1, Conlon, Fox, Pach, Sudakov, and Suk \cite{suk} established a quasi-polynomial bound for $R^{1,t}_3(s,n)$. In particular, $R^{1,t}_3(s,n) \leq 2^{(\log n)^C}$ where $C = C(t,s)$.   Combining Theorems \ref{main} and \ref{semirec} we have the following.

 \begin{corollary}\label{maincor}
For fixed integers $d,t\geq 1$, $k\geq 3$, and $s\geq k + 1$, we have

$$R^{d,t}_k(s,n) \leq \twr_{k-1}(n^{o(1)}).$$

\end{corollary}

\noindent The classic cups-caps construction of Erd\H os and Szekeres \cite{es} is an example of a planar point set with ${n + s - 4\choose s-2}$ elements and no $n$-cup and no $s$-cap.  This implies that $R^{d,t}_3(s,n) \geq \Omega(n^{s-2})$ for $d\geq 2$ and $t \geq 1$, and together with the semi-algebraic stepping-up lemma proven in \cite{suk} (see also \cite{mat4}) we have $R^{d,t}_k(s,n) \geq \twr_{k-2}(\Omega(n^{s/2^k}))$ for $s,d \geq 2^{k}$.

In Section \ref{secosh}, we give an application of Theorem \ref{main} to a recently studied problem on hyperplane arrangements in $\RR^d$.

\bigskip

\noindent \textbf{Monochromatic triangles.}  Let $R_2(s;m) = R_2(\underbrace{s,\ldots,s}_m)$ denote the smallest integer $N$ such that any $m$-coloring on the edges of the complete $N$-vertex graph contains a monochromatic clique of size $s$, that is, a set of $s$ vertices such that every pair from this set has the same color.  For the case $s = 3$, the Ramsey number $R_2(3;m)$ has received a lot of attention over the last 100 years due to its application in additive number theory \cite{schur} (more details are given in Section \ref{schursec}).  It is known (see \cite{fred,schur}) that

$$\Omega(3.19^{m}) \leq R_2(3;m) \leq O(m!).$$

Our next result states that we can improve the upper bound on $R_2(3;m)$ in our semi-algebraic setting.   More precisely, let $R^{d,t}_2(3;m)$ be the minimum integer $N$ such that every $N$-element point set $P$ in $\RR^d$ equipped with symmetric semi-algebraic relations $E_1,\ldots,E_m\subset {P\choose 2}$, such that each $E_i$ has complexity at most $t$ and ${P\choose 2} = E_1\cup \cdots \cup E_m$, contains three points such that every pair induced by them belongs to $E_i$ for some fixed $i$.

\begin{theorem}
\label{color}
For fixed $d,t\geq 1$ we have

$$  R^{d,t}_2(3;m) < 2^{O(m\log\log m)}.$$
\end{theorem}

\medskip

We also show that for fixed $d\geq 1$ and $t\geq 5000$, the function $R^{d,t}_2(3;m)$ does indeed grow exponentially in $m$.

\begin{theorem}\label{lowermulti}

For $d\geq 1$ and $t\geq 5000$ we have

$$ R^{d,t}_2(3;m)  \geq c(1681)^{m/7} \geq c(2.889)^m,$$

\noindent where $c$ is an absolute constant.

\end{theorem}

\medskip

\noindent \textbf{Organization.}  In the next two sections, we recall several old theorems on the arrangement of surfaces in $\RR^d$ and establish a result on point sets equipped with multiple binary relations.  In Section \ref{proofsection}, we combine the results from Sections \ref{cutsection} and \ref{multisection} to prove our main result, Theorem \ref{main}.  We discuss a short proof of our application in Section \ref{secosh}, and our results on monochromatic triangles in Section~\ref{triangles}. We conclude with some remarks.

We systemically omit floor and ceiling signs whenever they are not crucial for the sake of clarity of our presentation.  All logarithms are assumed to be base 2.

\section{Arrangement of surfaces in $\RR^d$}\label{cutsection}

In this section, we recall several old results on the arrangement of surfaces in $\RR^d$.  Let $f_1,\ldots, f_m$ be $d$-variate real polynomials of degree at most $t$, with zero sets $Z_1,\ldots, Z_m$, that is, $Z_i = \{x\in \RR^d: f_i(x) = 0\}$.  Set $\Sigma = \{Z_1,\ldots,Z_m\}$.  We will assume that $d$ and $t$ are fixed, and $m$ is some number tending to infinity.  A \emph{cell} in the arrangement $\mathcal{A}(\Sigma) = \bigcup_i Z_i$ is a relatively open connected set defined as follows.  Let $\approx$ be an equivalence relation on $\mathbb{R}^d$, where $x \approx y$ if $\{i: x \in Z_i\} = \{i:y \in Z_i\}$.  Then the cells of the arrangement $\mathcal{A}(\Sigma)$ are the connected components of the equivalence classes.  A vector $\sigma \in \{-1, 0, +1\}^m$ is a {\it sign pattern} of $f_1,\ldots, f_m$ if there exists an $x \in \mathbb{R}^d$ such that the sign of $f_j(x)$ is $\sigma_j$ for all $j = 1,\ldots, m$. The Milnor-Thom theorem (see \cite{basu, milnor, thom}) bounds the number of cells in the arrangement of the zero sets $Z_1,\ldots, Z_m$ and, consequently, the number of possible sign patterns.

\begin{theorem}[Milnor-Thom]\label{milnor}
Let $f_1,\ldots,f_m$ be $d$-variate real polynomials of degree at most $t$.  The number of cells in the arrangement of their zero sets $Z_1,\ldots,Z_m\subset \mathbb{R}^d$ and, consequently, the number of sign patterns of $f_1,\ldots,f_m$ is at most

$$\left(\frac{50mt}{d}\right)^d,$$

\noindent for $m \geq d \geq 1$.
\end{theorem}

While the Milnor-Thom Theorem bounds the number of cells in the arrangement $\mathcal{A}(\Sigma)$, the complexity of these cells may be very large (depending on $m$).  A long standing open problem is whether each cell can be further decomposed
into semi-algebraic sets\footnote{A real semi-algebraic set in $\RR^d$ is the locus of all points that satisfy a given finite Boolean combination of polynomial equations and inequalities in the $d$ coordinates.} with bounded description complexity (which depends only on $d$ and $t$), such that the total number of cells for the whole arrangement is still $O(m^d)$.  This can be done easily in dimension 2 by a result of Chazelle et al.~\cite{chazelle}.  Unfortunately in higher dimensions, the current bounds for this problem are not tight.  In dimension 3, Chazelle et al.~\cite{chazelle} established a near tight bound of $O(m^3\beta(m))$, where $\beta(m)$ is an extremally slowly growing function of $m$ related to the inverse Ackermann function.  For dimensions $d\geq 4$, Koltun \cite{koltun} established a general bound of $O(m^{2d-4 + \epsilon})$ for arbitrarily small constant $\epsilon$, which is nearly tight in dimension 4.  By combining these bounds with the standard theory of random sampling \cite{agarwal,shor,noga}, one can obtain the following result which is often referred to as the Cutting Lemma.  We say that the surface $Z_i = \{x \in \mathbb{R}^d: f_i(x) = 0\}$ \emph{crosses} the cell $\Delta\
\subset \mathbb{R}^d$ if $Z_i\cap \Delta \neq \emptyset$ and $Z_i$ does not fully contain $\Delta$.

\begin{lemma}[Cutting Lemma]
\label{cut2}
For $d,t \geq 1$, let $\Sigma$ be a family of $m$ algebraic surfaces (zero sets) in $\mathbb{R}^d$ of degree at most $t$.  Then for any $r > 0$, there exists a decomposition of $\mathbb{R}^d$ into at most $c_1r^{2d}$ relatively open connected sets (cells), where $c_1 = c_1(d,t) \geq 1$, such that each cell is crossed by at most $m/r$ surfaces from $\Sigma$.

\end{lemma}

As an application, we prove the following lemma (see \cite{mat2,chan} for a similar result when $\Sigma$ is a collection of hyperplanes).

\begin{lemma}\label{decomp}
For $d,t\geq 1$, let $P$ be an $N$-element point set in $\RR^d$ and let $\Sigma$ be a family of $m$ surfaces of degree at most $t$.  Then for any integer $\ell$ where $\log m < \ell < N/10$, we can find $\ell$ disjoint subsets $P_i$ of $P$ and $\ell$ cells $\Delta_i$, with $\Delta_i \supset P_i$, such that each subset $P_i$ contains at least $ N/(4\ell)$ points from $P$, and every surface in $\Sigma$ crosses at most $c_2\ell^{1 - 1/(2d)}$ cells $\Delta_i$, where $c_2 = c_2(d,t)$.
\end{lemma}

\begin{proof}

We first find $\Delta_1$ and $P_1$ as follows.  Let $\ell > \log m$ and let $c_1$ be as defined in Lemma \ref{cut2}.  Given a family $\Sigma$ of $m$ surfaces in $\RR^d$, we apply Lemma \ref{cut2} with parameter $r = \left(\ell/c_1\right)^{1/2d}$, and decompose $\RR^d$ into at most $\ell$ cells, such that each cell is crossed by at most $\frac{m}{(\ell/c_1)^{1/2d}}$ surfaces from $\Sigma$.  By the pigeonhole principle, there is a cell $\Delta_1$ that contains at least $N/\ell$ points from $P$.  Let $P_1$ be a subset of exactly $\lfloor N/\ell \rfloor$ points in $\Delta_1\cap P$.  Now for each surface from $\Sigma$ that crosses $\Delta_1$, we ``double it" by adding another copy of that surface to our collection.  This gives us a new family of surfaces $\Sigma_1$ such that

$$|\Sigma_1| \leq m + \frac{m}{(\ell/c_1)^{1/2d}} = m\left(1 + \frac{1}{(\ell/c_1)^{1/2d}}\right).$$

After obtaining subsets $P_1,\ldots,P_i$ such that $|P_j|= \lfloor \frac{N}{\ell}(1 - \frac{1}{\ell})^{j-1}\rfloor$ for $1\leq j \leq i$, cells $\Delta_1,\ldots,\Delta_i$, and a family of surfaces $\Sigma_i$ such that

$$|\Sigma_i| \leq   m\left(1 + \frac{1}{(\ell/c_1)^{1/2d}}\right)^i,$$

\noindent we obtain $P_{i + 1}$, $\Delta_{i + 1}$, $\Sigma_{i + 1}$ as follows.  Given $\Sigma_i$, we apply Lemma \ref{cut2} with the same parameter $r = \left(\ell/c_1\right)^{1/2d}$, and decompose $\RR^d$ into at most $\ell$ cells, such that each cell is crossed by at most $\frac{|\Sigma_i|}{(\ell/c_1)^{1/2d}}$ surfaces from $\Sigma_i$.  Let $P' = P\setminus(P_1\cup \cdots \cup P_i)$.  By the pigeonhole principle, there is a cell $\Delta_{i + 1}$ that contains at least

$$\begin{array}{ccl}
    \frac{|P'|}{\ell} & \geq & \left(N - \sum\limits_{j = 1}^i \frac{N}{\ell}(1 - \frac{1}{\ell})^{j - 1}\right)/\ell \\\\
      & = & \frac{N}{\ell}\left( 1- \frac{1}{\ell}\sum\limits_{j = 1}^i(1 - \frac{1}{\ell})^{j - 1}\right) \\\\
      &  =  & \frac{N}{\ell}\left(1 - \frac{1}{\ell}\right)^{i}
  \end{array}$$

\noindent points from $P'$.  Let $P_{i + 1}$ be a subset of exactly $\lfloor \frac{N}{\ell}\left(1 - 1/\ell\right)^{i} \rfloor$ points in $\Delta_{i + 1}\cap P'$.  Finally, for each surface from $\Sigma_i$ that crosses $\Delta_{i + 1}$, we ``double it" by adding another copy of that surface to our collection, giving us a new family of surfaces $\Sigma_{i + 1}$ such that

$$\begin{array}{ccl}
   |\Sigma_{i + 1}| & \leq & |\Sigma_i| + \frac{|\Sigma_i|}{(\ell/c_1)^{1/2d}} \\\\
      &  = &  |\Sigma_i|\left( 1 +\frac{1}{(\ell/c_1)^{1/2d}}\right)  \\\\
      & \leq & m\left(1 + \frac{1}{(\ell/c_1)^{1/2d}}\right)^{i + 1}.
  \end{array}$$

\noindent   Notice that $|P_i| \geq N/(4\ell)$ for $i \leq \ell$.  Once we have obtained subsets $P_1,\ldots,P_{\ell}$ and cell $\Delta_1,\ldots,\Delta_{\ell}$, it is easy to see that each surface in $\Sigma$ crosses at most $O(r^{1 - 1/2d})$ cells $\Delta_i$.  Indeed suppose $Z \in \Sigma$ crosses $\kappa$ cells.  Then by the arguments above, there must be $2^{\kappa}$ copies of $Z$ in $\Sigma_{\ell}$.  Hence we have

$$2^{\kappa} \leq  m\left(1 + \frac{1}{(\ell/c_1)^{1/2d}}\right)^{\ell} \leq me^{c_1\ell^{1  - 1/2d}}.$$

\noindent Since $\ell \geq \log m$, we have

$$\kappa \leq c_2\ell^{1 - 1/2d},$$

\noindent for sufficiently large $c_2 = c_2(d,t)$.\end{proof}

\section{Multiple binary relations}\label{multisection}

Let $P$ be a set of $N$ points in $\RR^d$, and let $E_1,\ldots,E_m\subset {P\choose 2}$ be binary semi-algebraic relations on $P$ such that $E_i$ has complexity at most $t$.  The goal of this section is to find a large subset $P'\subset P$ such that ${P'\choose 2}\cap E_i = \emptyset$ for all $i$, given that the clique number in the graphs $G_i = (P,E_i)$ are small.

First we recall a classic theorem of Dilworth (see also \cite{monot}).  Let $G =(V,E)$ be a graph whose vertices are ordered $V = \{v_1,\ldots,v_N\}$.  We say that $E$ is \emph{transitive} on $V$ if for  $1 \leq i_1 < i_2 < i_3\leq N$, $(v_{i_1},v_{i_2}), (v_{i_2},v_{i_3}) \in E$ implies that $(v_{i_1},v_{i_3}) \in E$.

\begin{theorem}[Dilworth]\label{dilworth}
Let $G = (V,E)$ be an $N$-vertex graph whose vertices are ordered $V = \{v_1,\ldots,v_N\}$, such that $E$ is transitive on $V$.  If $G$ has clique number $\omega$, then $G$ contains an independent set of order $N/\omega$.
\end{theorem}

\begin{lemma}\label{colors}
For integers $m \geq 2$ and $d,t \geq 1$, let $P$ be a set of $N$ points in $\mathbb{R}^d$ equipped with (symmetric) semi-algebraic relations $E_1,\ldots,E_m\subset{P\choose 2}$, where each $E_i$ has complexity at most $t$.  Then there is a subset $P'\subset P$ of size $N^{1/(c_3\log m)}$, where $c_3 = c_3(d,t)$, and a fixed ordering on $P'$ such that each relation $E_i$ is transitive on $P'$.

\end{lemma}

\begin{proof}

We proceed by induction on $N$.  Let $c_3$ be a sufficiently large number depending only on $d$ and $t$ that will be determined later.  For each relation $E_i \subset {P\choose 2}$, let $f_{i,1},\ldots,f_{i,t}$ be polynomials of degree at most $t$ and let $\Phi_i$ be a boolean function such that

$$(p,q) \in E_i \hspace{.5cm}\Leftrightarrow\hspace{.5cm} \Phi_i(f_{i,1}(p,q) \geq 0,\ldots,f_{i,t}(p,q) \geq 0) = 1.$$

For each $p \in P$, $i \in \{1,\ldots,m\}$, and $j \in \{1,\ldots,t\}$, we define the surface $Z_{p,i,j} =  \{x \in \RR^d: f_{i,j}(p,x) = 0 \}$.  Then let $\Sigma$ be the family of $Nmt$ surfaces in $\RR^d$ defined by

$$\Sigma = \{Z_{p,i,j} : p \in P, 1\leq i \leq m, 1\leq j \leq t\}.$$

  By applying Lemma \ref{cut2} to $\Sigma$ with parameter $r = (mt)^2$, there is a decomposition of $\RR^d$ into at most $c_1(mt)^{4d}$ cells such that each cell has the property that at most $N/(mt)$ surfaces from $\Sigma$ crosses it.  We note that $c_1 = c_1(d,t)$ is defined in Lemma \ref{cut2}.  By the pigeonhole principle, there is a cell $\Delta$ in the decomposition such that $|\Delta\cap P| \geq N/(c_1(mt)^{4d})$.  Set $P_1 = \Delta\cap P$.

  Let $P_2\subset P\setminus P_1$ be such that each point in $P_2$ gives rise to $mt$ surfaces that do not cross $\Delta$.  More precisely,

$$P_2 = \{p \in P\setminus P_1: Z_{p,i,j} \textnormal{ does not cross }\Delta, \forall i,j\}.$$

\noindent Since $m\geq 2$ by assumption, and $c_1 \geq 1$ from Lemma \ref{cut2}, we have $$|P_2| \geq N - \frac{N}{mt} - \frac{N}{c_1(mt)^{4d}} \geq \frac{N}{4}.$$  We fix a point $p_0 \in P_1$.  Then for each $q\in P_2$, let $\sigma(q) \in \{-1,0,+1\}^{mt}$ be the sign pattern of the $(mt)$-tuple $(f_{1,1}(p_0,q),f_{1,2}(p_0,q),\ldots,f_{m,t}(p_0,q))$.  By Theorem \ref{milnor}, there are at most $\left(\frac{50mt^2}{d}\right)^d$ distinct sign vectors $\sigma$.  By the pigeonhole principle, there is a subset $P_3 \subset P_2$ such that $$|P_3| \geq \frac{|P_2|}{(50/d)^dm^dt^{2d}},$$ and for any two points $q,q' \in P_3$, we have $\sigma(q) = \sigma(q')$.  That is, $q$ and $q'$ give rise to vectors with the same sign pattern.  Therefore, for any $p,p' \in P_1$ and $q,q' \in P_3$, we have $(p,q) \in E_i$ if and only if $(p',q') \in E_i$, for all $i \in \{1,\ldots,m\}$.

Let $c_4 = c_4(d,t)$ be sufficiently large such that $|P_1|,|P_3| \geq \frac{N}{c_4m^{4d}}.$  By the induction hypothesis, we can find subsets $P_4 \subset P_1, P_5\subset P_3$, such that

$$|P_4|,|P_5| \geq \left(\frac{N}{c_4m^{4d}}\right)^{\frac{1}{c_3\log m}}\geq \frac{N^{\frac{1}{c_3\log m}}}{2},$$

\noindent where $c_3 = c_3(d,t)$ is sufficiently large, and there is an ordering on $P_4$ (and on $P_5)$ such that each $E_i$ is transitive on $P_4$ (and on $P_5$).  Set $P' = P_4\cup P_5$, which implies $|P'| \geq N^{\frac{1}{c_3\log m}}$.  We will show that $P'$ has the desired properties.  Let $\pi$ and $\pi'$ be the orderings on $P_4$ and $P_5$ respectively, such that $E_i$ is transitive on $P_4$ and on $P_5$, for every $i\in \{1,\ldots,m\}$.  We order the elements in $P' = \{p_1,\ldots,p_{|P'|}\}$ by using $\pi$ and $\pi'$, such that all elements in $P_5$ come after all elements in $P_4$.

In order to show that $E_i$ is transitive on $P'$, it suffices to examine triples going across $P_4$ and $P_5$.  Let $p_{j_1},p_{j_2} \in P_4$ and $p_{j_3} \in P_5$ such that $j_1 < j_2 < j_3$.  By construction of $P_4$ and $P_5$, if $(p_{j_1},p_{j_2}),(p_{j_2},p_{j_3}) \in E_i$, then we have $(p_{j_1},p_{j_3}) \in E_i$.  Likewise, suppose $p_{j_1}\in P_4$ and $p_{j_2},p_{j_3} \in P_5$.  Then again by construction of $P_4$ and $P_5$, if $(p_{j_1},p_{j_2}),(p_{j_2},p_{j_3}) \in E_i$, then we have $(p_{j_1},p_{j_3}) \in E_i$.  Hence $E_i$ is transitive on $P'$, for all $i \in \{1,\ldots,m\}$, and this completes the proof.\end{proof}

 By combining the two previous results, we have the following.

\begin{lemma}\label{apply}
For $m \geq 2$ and $d,t \geq 1$, let $P$ be a set of $N$ points in $\mathbb{R}^d$ equipped with (symmetric) semi-algebraic relations $E_1,\ldots,E_m\subset {P\choose 2}$, where each $E_i$ has complexity at most $t$.  If graph $G_i = (P,E_i)$ has clique number $\omega_i$, then there is a subset $P'\subset P$ of size $\frac{N^{1/(c_3\log m)}}{\omega_1\cdots \omega_m}$, where $c_3 = c_3(d,t)$ is defined above, such that ${P'\choose 2}\cap E_i = \emptyset$ for all $i$.

\end{lemma}

\begin{proof}
By applying Lemma \ref{colors}, we obtain a subset $P_1 \subset P$ of size $N^{\frac{1}{c_3\log m}}$, and an ordering on $P_1$ such that $E_i$ is transitive on $P_1$ for all $i$.  Then by an $m$-fold application of Theorem \ref{dilworth}, the statement follows.\end{proof}

\section{Proof of Theorem \ref{main}}\label{proofsection}

Let $P$ be a point set in $\RR^d$ and let $E\subset {P\choose 3}$ be a semi-algebraic relation on $P$.  We say that $(P,E)$ is $K_s^{(3)}$-free if every collection of $s$ points in $P$ contains a triple not in $E$.  Suppose we have $\ell$ disjoint subsets $P_1,\ldots,P_{\ell} \subset P$.  For $1 \leq i_1 < i_2 < i_3 \leq \ell$, we say that the triple $(P_{i_1},P_{i_2},P_{i_3})$ is \emph{homogeneous} if $(p_1,p_2,p_3) \in E$ for all $p_1\in P_{i_1}, p_2 \in P_{i_2},p_3 \in P_{i_3}$, or $(p_1,p_2,p_3) \not\in E$ for all $p_1\in P_{i_1}, p_2 \in P_{i_2},p_3 \in P_{i_3}$.  For $p_1,p_2 \in P_1\cup \cdots \cup P_{\ell}$ and $i \in \{1,\ldots,\ell\}$, we say that the triple $(p_1,p_2,i)$ is \emph{good}, if $(p_1,p_2,p_3) \in E$ for all $p_3 \in P_i$, or $(p_1,p_2,p_3) \not\in E$ for all $p_3 \in P_i$.  We say that the triple $(p_1,p_2,i)$ is \emph{bad} if $(p_1,p_2,i)$ is not good and $p_1,p_2 \not\in P_i$.

\begin{lemma}\label{decomp2}
Let $P$ be a set of $N$ points in $\RR^d$ and let $E\subset {P\choose 3}$ be a (symmetric) semi-algebraic relation on $P$ such that $E$ has complexity at most $t$.  Then for $r = \frac{N^{1/(30d)}}{tc_2}$, where $c_2$ is defined in Lemma \ref{decomp}, there are disjoint subsets $P_1,\ldots,P_{r}\subset P$ such that

\begin{enumerate}

\item $|P_i| \geq \frac{N^{1/(30d)}}{tc_2}$,

\item all triples $(P_{i_1},P_{i_2},P_{i_3})$, $1\leq i_1 < i_2 <i_3\leq r$, are homogeneous, and

\item all triples $(p,q,i)$, where $i\in \{1,\ldots,r\}$ and $p,q \in (P_1\cup \cdots \cup P_r)\setminus P_i$, are good.

\end{enumerate}

\begin{proof}

We can assume that $N > (tc_2)^{30d}$, since otherwise the statement is trivial. Since $E$ is semi-algebraic with complexity $t$, there are polynomials $f_1,\ldots,f_t$ of degree at most $t$, and a Boolean function $\Phi$ such that

$$(p_1,p_2,p_3) \in E\hspace{.5cm}\Leftrightarrow\hspace{.5cm} \Phi(f_1(p_1,p_2,p_3) \geq 0,\ldots,f_t(p_1,p_2,p_3) \geq 0)  = 1.$$

\noindent For each $p,q \in P$ and $i \in \{1,\ldots,t\}$, we define the surface $Z_{p,q,i} = \{x \in \RR^d: f_i(p,q,x) = 0\}$.  Then we set

$$\Sigma = \{Z_{p,q,i} : p,q \in P, 1\leq i\leq t\}.$$

Thus we have $|\Sigma| = N^2t$.  Next we apply Lemma \ref{decomp} to $P$ and $\Sigma$ with parameter $\ell = \sqrt{N}$, and obtain subsets $Q_1,\ldots,Q_{\ell}$ and cells $\Delta_1,\ldots,\Delta_{\ell}$, such that $Q_i\subset \Delta_i$, $|Q_i| = \lfloor\sqrt{N}/4\rfloor$, and each surface in $\Sigma$ crosses at most $c_2N^{1/2 - 1/(4d)}$ cells $\Delta_i$.  We note that $c_2 = c_2(d,t)$ is defined in Lemma \ref{decomp} and $\sqrt{N} \geq \log(tN^2)$.  Set $Q = Q_1\cup \cdots \cup Q_{\ell}$.  Each pair $(p,q) \in {Q\choose 2}$ gives rise to $2t$ surfaces in $\Sigma$.  By Lemma \ref{decomp}, these $2t$ surfaces cross in total at most $2tc_2N^{1/2 - 1/(4d)}$ cells $\Delta_i$.  Hence there are at most $2tc_2N^{5/2 - 1/(4d)}$ bad triples of the form $(p,q,i)$, where $i \in \{1,\ldots,\sqrt{N}\}$ and $p,q \in Q\setminus Q_i$.  Moreover, there are at most $2tc_2N^{2 - 1/(4d)}$ bad triples $(p,q,i)$, where both $p$ and $q$ lie in the same part $Q_j$ and $j \neq i$.

We uniformly at random pick $r = \frac{N^{1/(30d)}}{tc_2}$ subsets (parts) from the collection $\{Q_1,\ldots,Q_{\ell}\}$, and $r$ vertices from each of the subsets that were picked.  For a bad triple $(p,q,i)$ with $p$ and $q$ in distinct subsets, the probability that $(p,q,i)$ survives is at most

$$\left(\frac{r}{\sqrt{N}}\right)^3\left(\frac{r}{\sqrt{N}/4}\right)^2  = \frac{16}{(tc_2)^5}N^{1/(6d) - 5/2}.$$

\noindent For a bad triple $(p,q,i)$ with $p,q$ in the same subset $Q_j$, where $j \neq i$, the probability that the triple $(p,q,i)$ survives is at most

$$\left(\frac{r}{\sqrt{N}}\right)^2\left(\frac{r}{\sqrt{N}/4}\right)^2 = \frac{16}{(tc_2)^4}N^{2/(15d) - 2}.$$

\noindent Therefore, the expected number of bad triples in our random subset is at most

$$\left(\frac{16}{(tc_2)^5}N^{1/(6d) - 5/2}\right)\left(tc_2N^{5/2 - 1/(4d)}\right) + \left(\frac{16}{(tc_2)^4}N^{2/(15d) - 2}\right)\left(tc_2N^{2 - 1/(4d)}\right) < 1.$$

\noindent  Hence we can find disjoint subsets $P_1,\ldots,P_r$, such that $|P_i| \geq r = \frac{N^{1/(30d)}}{tc_2}$, and there are no bad triples $(p,q,i)$, where $i \in \{1,\ldots,r\}$ and $p,q \in (P_1\cup \cdots \cup P_{r})\setminus P_i$.

It remains to show that every triple $(P_{i_1},P_{i_2},P_{i_3})$ is homogeneous for $1\leq i_1 < i_2 < i_3\leq r$.  Let $p_1,\in P_{i_1}, p_2 \in P_{i_2}, p_3 \in P_{i_3}$ and suppose $(p_1,p_2,p_3) \in E$.  Then for any choice $q_1,\in P_{i_1}, q_2 \in P_{i_2}, q_3 \in P_{i_3}$, we also have $(q_1,q_2,q_3) \in E$.  Indeed, since the triple $(p_1,p_2,i_3)$ is good, this implies that $(p_1,p_2,q_3) \in E$.  Since the triple $(p_1,q_3,i_2)$ is also good, we have $(p_1,q_2,q_3) \in E$.  Finally since $(q_2,q_3,i_1)$ is good, we have $(q_1,q_2,q_3) \in E$.   Likewise, if $(p_1,p_2,p_3) \not\in E$, then $(q_1,q_2,q_3) \not\in E$ for any $q_1,\in P_{i_1}, q_2 \in P_{i_2}, q_3 \in P_{i_3}$.\end{proof}

\end{lemma}

We are finally ready to prove Theorem \ref{main}, which follows immediately from the following theorem.

\begin{theorem}
Let $P$ be a set of $N$ points in $\RR^d$ and let $E\subset {P\choose 3}$ be a (symmetric) semi-algebraic relation on $P$ such that $E$ has complexity at most $t$.  If $(P,E)$ is $K^{(3)}_s$-free, then there exists a subset $P'\subset P$ such that ${P'\choose 3}\cap E = \emptyset$ and

$$|P'| \geq 2^{\frac{(\log \log N)^2}{c^s\log\log\log N}},$$

\noindent where $c = c(d,t)$.
\end{theorem}

\begin{proof}

The proof is by induction on $N$ and $s$.  The base cases are $s = 3$ or $N \leq (100tc_2)^{30d}$, where $c_2$ is defined in Lemma \ref{decomp}.  When $N \leq (100tc_2)^{30d}$, the statement holds trivially for sufficiently large $c = c(d,t)$.  If $s = 3$, then again the statement follows immediately by taking $P' = P$.

Now assume that the statement holds if $s ' \leq s, N' \leq N$ and not both inequalities are equalities.  We apply Lemma \ref{decomp2} to $(P,E)$ and obtain disjoint subsets $P_1,\ldots,P_r$, where $r  = \frac{N^{1/(30d)}}{tc_2}$, such that $|P_i| \geq \frac{N^{1/(30d)}}{tc_2}$, every triple of parts $(P_{i_1},P_{i_2},P_{i_3})$ is homogeneous, and every triple $(p,q,i)$ is good where $i \in \{1,\ldots,r\}$ and $p,q \in (P_1\cup\cdots\cup P_r) \setminus P_i$.

Let $P_0$ be the set of $\frac{N^{1/(30d)}}{tc_2}$ points obtained by selecting one point from each $P_i$.  Since $(P_0,E)$ is $K^{(3)}_s$-free, we can apply the induction hypothesis on $P_0$, and find a set of indices $I = \{i_1,\ldots,i_m\}$ such that

$$\log |I| \geq  \frac{\left(\log\log \frac{N^{1/(30d)}}{tc_2}\right)^2}{{c^s\log\log \log \frac{N^{1/(30d)}}{tc_2}} }  \geq (1/2)\log \log N,$$

\noindent and for every triple $i_1 < i_2 < i_3$ in $I$ all triples with one point in each $P_{i_j}$ do not satisfy $E$.  Hence we may assume $m = \sqrt{\log N}$, and let $Q_j = P_{i_j}$ for $1\leq j \leq m$.

For each subset $Q_i$, we define binary semi-algebraic relations $E_{i,j}\subset {Q_i\choose 2}$, where $j \neq i  $, as follows. Since $E\subset {P\choose 3}$ is semi-algebraic with complexity $t$, there are $t$ polynomials $f_1,\ldots,f_t$ of degree at most $t$, and a Boolean function $\Phi$ such that $(p_1,p_2,p_3) \in E$ if and only if

$$\Phi(f_1(p_1,p_2,p_3) \geq 0,\ldots,f_t(p_1,p_2,p_3) \geq 0)  = 1.$$

\noindent Fix a point $q_0 \in Q_j$, where $j \neq i$.  Then for $p_1,p_2 \in Q_i$, we have $(p_1,p_2) \in E_{i,j}$ if and only if

$$\Phi(f_1(p_1,p_2,q_0) \geq 0,\ldots,f_t(p_1,p_2,q_0) \geq 0)  = 1.$$

 Suppose there are $2^{(\log N)^{1/4}}$ vertices in $Q_i$ that induce a clique in the graph $G_{i,j} = (Q_i,E_{i,j})$.  Then these vertices would induce a $K_{s-1}^{(3)}$-free subset in the original (hypergraph) $(P,E)$.  By the induction hypothesis, we can find a subset $Q_i' \subset Q_i$ such that

  $$|Q'_i| \geq 2^{\frac{((1/4)\log\log N)^2}{c^{s-1}\log\log\log N}} \geq 2^{\frac{(\log \log N)^2}{c^s\log\log\log N}},$$

  \noindent for sufficiently large $c$, such that ${Q'_i\choose 3}\cap E = \emptyset$ and we are done.  Hence we can assume that each graph $G_{i,j} = (Q_i,E_{i,j})$ has clique number at most $2^{(\log n)^{1/4}}$.  By applying Lemma \ref{apply} to each $Q_i$, where $Q_i$ is equipped with $m-1$ semi-algebraic relations  $E_{i,j}$, $j \neq i$, we can find subsets $T_i\subset Q_i$ such that

   $$|T_i| \geq \frac{|Q_i|^{1/(c_3\log m)}}{2^{(\log N)^{1/4}\sqrt{\log N}}} = \frac{2^{\frac{\log N}{30dc_3\log(\sqrt{\log N})}}}{(tc_2)^{1/c_3\log m}2^{(\log N)^{3/4}}} \geq 2^{\frac{\log N}{c_5\log\log N}},$$

   \noindent where $c_5  = c_5(d,t)$, and ${T_i\choose 2} \cap E_j = \emptyset$ for all $j \neq i$.   Therefore, we now have subsets $T_1,\ldots,T_m$, such that

   \begin{enumerate}

   \item $m = \sqrt{\log N}$,

   \item for any triple $(T_{i_1},T_{i_2},T_{i_3})$, $1 \leq i_1 < i_2 < i_3\leq m$, every triple with one vertex in each $T_{i_j}$ is not in $E$,

   \item for any pair $(T_{i_1},T_{i_2})$, $1 \leq i_1 < i_2 \leq m$, every triple with two vertices in $T_{i_1}$ and one vertex in $T_{i_2}$ is not in $E$, and every triple with two vertices in $T_{i_2}$ and one vertex in $T_{i_1}$ is also not in $E$.

   \end{enumerate}

\noindent By applying the induction hypothesis to each $(T_i,E)$,  we obtain a collection of subsets $U_i \subset T_i$ such that

$$\log |U_i| \geq  \frac{\left(\log \left(\frac{\log N}{c_5\log\log N}\right)\right)^2}{c^s\log\log\left(\frac{\log N}{c_5\log\log N} \right)} \geq  \frac{(\log\log N - \log(c_5\log\log N))^2}{c^s\log\log\log N},$$

\noindent and ${U_i\choose 3}\cap E = \emptyset$.  Let $P' = \bigcup\limits_{i = 1}^m U_i$.  Then by above we have ${P'\choose 3}\cap E = \emptyset$ and

\begin{eqnarray*}
\log |P'| & \geq & \frac{(\log\log N - \log(c_5\log\log N))^2}{c^s\log\log\log N} +  \frac{1}{2}\log\log N \\\\
 & \geq & \frac{(\log\log N)^2  - 2 (\log\log N)\log(c_5\log\log N) + (\log(c_5\log\log N))^2}{c^s\log\log\log N} + \frac{1}{2}\log\log N \\\\
  & \geq & \frac{(\log\log N)^2}{c^s\log\log \log N},
 \end{eqnarray*}

\noindent for sufficiently large $c = c(d,t)$.\end{proof}

\section{Application: One-sided hyperplanes}\label{secosh}

Let us consider a finite set $H$ of hyperplanes in $\RR^d$ in general position, that is, every $d$ members in $H$ intersect at a distinct point.  Let $OSH_d(s,n)$ denote the smallest integer $N$ such that every set $H$ of $N$ hyperplanes in $\mathbb{R}^d$ in general position contains $s$ members $H_1$ such that the vertex set of the arrangement of $H_1$ lies above the $x_d = 0$ hyperplane, or contains $n$ members $H_2$ such that the vertex set of the arrangement of $H_2$ lies below the $x_d = 0$ hyperplane.

In 1992, Matou\v{s}ek and Welzl \cite{welzl} observed that $OSH_2(s,n) = (s-1)(n-1) + 1$. Dujmovi\'c and Langerman \cite{duj} used the existence of $OSH_d(n,n)$ to prove a ham-sandwich cut theorem for hyperplanes.  Again by adapting the Erd\H os-Rado argument, Conlon et al.~\cite{suk} showed that for $d\geq 3$,

\begin{equation}\label{oldosh}
OSH_d(s,n) \leq \twr_{d-1}(c_6sn\log n),
\end{equation}

\noindent where $c_6$ is a constant that depends only on $d$.  See Eli\'a\v{s} and Matou\v{s}ek \cite{matousek} for more related results, including lower bound constructions.

Since each hyperplane $h_i\in H$ is specified by the linear equation

$$a_{i,1}x_1 + \cdots + a_{i,d}x_d = b_i,$$

\noindent we can represent $h_i\in H$ by the point $h^{\ast}_i \in\mathbb{R}^{d + 1}$ where $h^{\ast}_i = (a_{i,1},\ldots,a_{i,d},b_i)$ and let $P = \{h^{\ast}_i:h_i \in H\}$.  Then we define a relation $E\subset {P\choose d}$ such that $(h^{\ast}_{i_1},\ldots,h^{\ast}_{i_d}) \in E$ if and only if $h_{i_1}\cap\cdots\cap h_{i_d}$ lies above the hyperplane $x_d = 0$ (i.e. the $d$-th coordinate of the intersection point is positive).  Clearly, $E$ is a semi-algebraic relation with complexity at most $t = t(d)$.  Therefore, as an application of Theorem \ref{main} and Corollary \ref{maincor}, we make the following improvement on (\ref{oldosh}).

\begin{theorem}

For fixed $s \geq 4$, we have $OSH_3(s,n) \leq 2^{n^{o(1)}}$.  For fixed $d\geq 4$ and $s\geq d+1$, we have

$$OSH_d(s,n) \leq \twr_{d-1}(n^{o(1)}).$$

\end{theorem}

\section{Monochromatic triangles}\label{triangles}

In this section, we will prove Theorem \ref{color}.

\begin{proof}[Proof of Theorem \ref{color}]  We proceed by induction on $m$.  The base case when $m = 1$ is trivial.  Now assume that the statement holds for $m' < m$.  Set $N = 2^{cm\log\log m}$, where $c = c(d,t)$ will be determined later, and let $E_1,\ldots,E_m\subset {P\choose 2}$ be semi-algebraic relations on $P$ such that ${P\choose 2} = E_1\cup \cdots \cup E_m$, and each $E_i$ has complexity at most $t$.  For the sake of contradiction, suppose $P$ does not contain three points such that every pair of them is in $E_i$ for some fixed $i$.

For each relation $E_i$, there are $t$ polynomials $f_{i,1},\ldots,f_{i,t}$ of degree at most $t$, and a Boolean function $\Phi_i$ such that

$$(p,q) \in E_i \hspace{.5cm}\Leftrightarrow\hspace{.5cm} \Phi_i(f_{i,1}(p,q) \geq 0,\ldots,f_{i,t}(p,q)\geq 0) = 1.$$

For $ 1 \leq i \leq m,  1\leq j \leq t, p \in P$, we define the surface $Z_{i,j,p} = \{ x \in \RR^d: f_{i,j}(p,x)  = 0\}$, and let $$\Sigma = \{ Z_{i,j,p} : 1 \leq i \leq m,  1\leq j \leq t, p \in P\}.$$

\noindent Hence $|\Sigma| = mtN$.  We apply Lemma \ref{cut2} to $\Sigma$ with parameter $r = 2tm$, and decompose $\RR^d$ into $c_1(2tm)^{2d}$ regions $\Delta_i$, where $c_1 = c_1(t,d)$ is defined in Lemma \ref{cut2}, such that each region $\Delta_i$ is crossed by at most $tmN/r = N/2$ members in $\Sigma$.  By the pigeonhole principle, there is a region $\Delta\subset \RR^d$, such that $|\Delta\cap P| \geq \frac{N}{c_1(2tm)^{2d}}$, and at most $N/2$ members in $\Sigma$ crosses $\Delta$.  Let $P_1$ be a set of exactly $\left\lfloor \frac{N}{c_1(2tm)^{2d}}\right\rfloor$ points in $P\cap \Delta$, and let $P_2$ be the set of points in $P\setminus P_1$ that do not give rise to a surface that crosses $\Delta$.  Hence

$$|P_2| \geq N - \frac{N}{c_1(2tm)^{2d}} - \frac{N}{2} \geq \frac{N}{4}.$$

Therefore, each point $p \in P_2$ has the property that $p\times P_1 \subset E_i$ for some fixed $i$.  We define the function $\chi:P_2\rightarrow \{1,\ldots,m\}$, such that $\chi(p) = i$ if and only if $p\times P_1 \subset E_i$.  Set $I = \{\chi(p): p \in P_2\}$ and $m_0 = |I|$, that is, $m_0$ is the number of distinct relations (colors) between the sets $P_1$ and $P_2$.  Now the proof falls into 2 cases.

\medskip

\noindent \emph{Case 1.}  Suppose $m_0 > \log m$.  By the assumption, every pair of points in $P_1$ is in $E_i$ where $i \in \{1,\ldots,m\}\setminus I$.  By the induction hypothesis, we have

$$\frac{2^{cm\log\log m}}{c_1(2tm)^{2d}} \leq |P_1|\leq 2^{c(m - m_0)\log\log m}.$$

\noindent Hence

$$cm_0\log\log m \leq \log(c_1(2tm)^{2d}) \leq 2d\log (c_12tm),$$

\noindent which implies

$$m_0 \leq \frac{2d\log(c_12tm)}{c\log\log m},$$

\noindent and we have a contradiction for sufficiently large $c = c(d,t)$.

\medskip

\noindent \emph{Case 2.}  Suppose $m_0 \leq \log m$.  By the pigeonhole principle, there is a subset $P_3\subset P_2$, such that $|P_3| \geq \frac{N}{4m_0}$ and $P_1\times P_3\subset E_i$ for some fixed $i$. Hence every pair of points $p,q \in P_3$ satisfies $(p,q) \not\in E_i$, for some fixed $i$.  By the induction hypothesis, we have

$$\frac{2^{cm\log\log m}}{4m_0} \leq |P_3| \leq 2^{c(m-1)\log\log m}.$$

\noindent Therefore

$$c\log\log m \leq \log(4m_0) \leq \log(4\log(m)),$$

\noindent which is a contradiction since $c$ is sufficiently large.  This completes the proof of Theorem \ref{color}.\end{proof}

 We note that in \cite{fps15}, Fox, Pach, and Suk extended the arguments above to show that $R_2^{d,t}(s;m) \leq 2^{O(sm\log\log m)}$.

\subsection{Lower bound construction and Schur numbers}\label{schursec}

Before we prove Theorem \ref{lowermulti}, let us recall a classic theorem of Schur \cite{schur} which is considered to be one of the earliest applications of Ramsey Theory.  A set $P\subset \RR$ is said to be \emph{sum-free} if for any two (not necessarily distinct) elements $x,y \in P$, their sum $x+y$ is not in $P$.  The Schur number $S(m)$ is defined to be the maximum integer $N$ for which the integers $\{1,\ldots,N\}$ can be partitioned into $m$ sum-free sets.

Given a partition $\{1,\ldots,N\}  =P_1\cup\cdots\cup P_m$ into $m$ parts such that $P_i$ is sum-free, we can define an $m$-coloring on the edges of a complete $(N+1)$-vertex graph which does not contain a monochromatic triangle as follows.  Let $V = \{1,\ldots,N+1\}$ be the vertex set, and we define the coloring $\chi:{V\choose 2} \rightarrow m$ by $\chi(x,y) = i$ iff $|x-y| \in P_{i}$.  Now suppose for the sake of contradiction that there are vertices $x,y,z$ that induce a monochromatic triangle, say with color $i$, such that $x < y < z$.  Then we have $y-x, z-y, z-x \in P_i$ and $(y -x) + (z - y) = (z - x)$, which is a contradiction since $P_i$ is sum free.  Therefore $S(m) < R_2(3;m)$.

Since Schur's original 1916 paper, the lower bound on $S(m)$ has been improved by several authors \cite{ab1,ab2,ex}, and the current record of $S(m)  \geq \Omega(3.19^m)$ is due to Fredricksen and Sweet \cite{fred}.  Their lower bound follows by computing $S(6)\geq 538$, and using the recursive formula

$$S(m)\geq c_{\ell}(2S(\ell) + 1)^{m/\ell},$$

 \noindent which was established by Abbott and Hanson \cite{ab2}.  Fredricksen and Sweet also computed $S(7) \geq 1680$, which we will use to prove Theorem \ref{lowermulti}.

\begin{lemma}\label{construct2}
For each integer $\ell \geq 1$, there is a set $P_{\ell}$ of $(1681)^{\ell}$ points in $\RR$ equipped with semi-algebraic relations $E_{1},\ldots,E_{7\ell} \subset {P_{\ell}\choose 2}$, such that

\begin{enumerate}

\item $E_{1}\cup \cdots \cup E_{7\ell} = {P_{\ell}\choose 2}$,

\item $E_{i}$ has complexity at most 5000,

\item $E_{i}$ is translation invariant, that is, $(x,y) \in E_i$ iff $(x + C,y + C) \in E_i$, and

\item the graph $G_{\ell,i} = (P_{\ell},E_{i})$ is triangle free for all $i$.

\end{enumerate}

\end{lemma}

\begin{proof}
We start be setting $P_{1} =\{1,2,\ldots,1681\}$.  By \cite{fred}, there is a partition on $\{1,\ldots,1680\} = A_1\cup \cdots \cup A_7$ into seven parts, such that each $A_i$ is sum-free.  For $i \in \{1,\ldots,7\}$, we define the binary relation $E_{i}$ on $P_1$ by

$$(x,y) \in E_{i}\hspace{.5cm}\Leftrightarrow\hspace{.5cm}(1 \leq |x- y| \leq 1680) \wedge(|x-y| \in A_i).$$

\noindent Since $|A_i| \leq 1680$, $E_{i}$ has complexity at most 5000.  By the arguments above, the graph $G_{1,i} = (P_{1},E_{i})$ is triangle free for all $i \in \{1,\ldots,7\}$. In what follows, we blow-up this construction so that the statement holds.

Having defined $P_{\ell - 1}$ and $E_{1},....,E_{7\ell-7}$, we define $P_{\ell}$ and $E_{\ell -6},\ldots,E_{\ell }$ as follows.  Let $C  = C(\ell)$ be a very large constant, say $C > \left(5000 \cdot \max\{P_{\ell - 1}\}\right)^2$.  We construct 1681 translated copies of $P_{\ell -1}$, $Q_i = P_{\ell - 1} + iC$ for $1\leq i \leq 1681$, and set $P_{\ell} = Q_1\cup \cdots \cup Q_{1681}$.  For $1\leq j \leq 7$, we define the relation $E_{\ell - 7 + j}$ by

$$(x,y) \in E_{\ell -7 + j}\hspace{.5cm}\Leftrightarrow\hspace{.5cm}(C/2 \leq |x-y| \leq 1682C) \wedge(\exists z \in A_j: ||x - y|/C  - z| < 1/1000).$$

Clearly $E_1,\ldots,E_{7\ell}$ satisfy properties (1), (2), and (3).  The fact that $G_{\ell,i} = (P_{\ell},E_i)$ is triangle follows from the same argument as above.\end{proof}

Theorem \ref{lowermulti} immediately follows from Lemma \ref{construct2}.

\section{Concluding remarks}

\textit{1.}  We showed that given an $N$-element point set $P$ in $\RR^d$ equipped with a semi-algebraic relation $E\subset{P\choose 3}$, such that $E$ has complexity at most $t$ and $(P,E)$ is $K^{(3)}_s$-free, then there is a subset $P'\subset P$ such that $|P'| \geq 2^{(\log\log N)^2/(c^s \log\log\log N)}$ and ${P'\choose 3} \cap E = \emptyset$.  In \cite{suk}, Conlon et al.~conjectured that one can find a much larger ``independent set".  More precisely, they conjectured that there is a constant $\epsilon = \epsilon(d,t,s)$ such that $|P'| \geq N^{\epsilon}$.  Perhaps an easier task would be to find a large subset $P'$ such that $E$ is \emph{transitive} on $P'$, that is, there is an ordering on $P' = \{p_1,\ldots,p_m\}$ such that for $1\leq i_1 < i_2 < i_3 < i_4 \leq m$, $(p_{i_1},p_{i_2},p_{i_3}), (p_{i_2},p_{i_3},p_{i_4}) \in E$ implies that $(p_{i_1},p_{i_2},p_{i_4}),(p_{i_1},p_{i_3},p_{i_4}) \in E$.

\bigskip

\noindent\textit{2. Off diagonal Ramsey numbers for binary semi-algebraic relations.}  As mentioned in the introduction, we have $R_2(s,n) \leq O(n^{s-1})$.  It would be interesting to see if one could improve this upper bound in the semi-algebraic setting.  That is, for fixed integers $t\geq 1$ and $d\geq s \geq 3$, is there an constant $\epsilon = \epsilon(d,t,s)$ such that $R^{d,t}_2(s,n) \leq O(n^{s-1- \epsilon})$?  For $d < s$, it is likely that such an improvement can be made using Lemma \ref{cut2}.

\bigskip

\noindent \textit{3. Low complexity version of Schur's Theorem.}  We say that the subset $P \subset \{1,\ldots,N\}$ has complexity $t$ if there are $t$ intervals $I_1,\ldots,I_t$ such that $P = \{1,\ldots,N\}\cap (I_1\cup \cdots \cup I_t)$.  Let $S_t(m)$ be the maximum integer $N$ for which the integers $\{1,\ldots,N\}$ can be partitioned into $m$ sum-free parts, such that each part has complexity at most $t$.  By following the proof of Theorem \ref{color}, one can show that $S_t(m) \leq 2^{m\log\log 2t}$.

\end{document}